\documentclass[10pt,a4paper]{article}
\usepackage{amsmath, amssymb, amsthm}
\usepackage{geometry}
\usepackage[utf8]{inputenc}
\usepackage{hyperref}
\usepackage{booktabs}

\usepackage{amsfonts}
\usepackage{amscd}
\usepackage{latexsym,mathtools}

\usepackage{authblk}
\usepackage{orcidlink}

\newtheorem{theorem}{Theorem}
\newtheorem{lemma}[theorem]{Lemma}

\newtheorem{definition}[theorem]{Definition}
\newtheorem{example}[theorem]{Example}
\newtheorem{remark}[theorem]{Remark}
\newtheorem{corollary}[theorem]{Corollary}

\newcommand{\GreenEst}{\frac{(\eta-\zeta)^2}{8}}
\newcommand{\norm}[1]{\left\|#1\right\|_\infty}

\newcommand{\diff}{\mathrm{d}}
\newcommand{\Cone}{C^1[\zeta,\eta]}
\newcommand{\Czero}{C[\zeta,\eta]}

\title{A Restatement of Fixed Point Existence: From Banach to  Singh Contractions}

\author[1]{Nicola Fabiano\, \orcidlink{0000-0003-1645-2071}}
\affil[1]{``Vin\v{c}a'' Institute of Nuclear Sciences - National 
Institute of the Republic of Serbia, University of Belgrade, Mike Petrovi\'{c}a 
Alasa 12--14, 11351 Belgrade, Serbia; nicola.fabiano@gmail.com}
\author[2]{Zouaoui Bekri\, \orcidlink{0000-0002-2430-6499}}
\affil[2]{Laboratory of fundamental and applied mathematics,
University of Oran 1, Ahmed Ben Bella,
Es-senia, 31000 Oran
Department of Sciences and Technology,
Institute of Sciences, Nour-Bachir University
Center, El-Bayadh, 32000
Algeria; z.bekri@cu-elbayadh.dz}
\date{}

\begin{document}
\maketitle

%%%%%%%%%%%%%%%%%%%%%%%%%%%%%%%%%%%%%%%%%%%%%

%%%%%%%%%%%%%%%%%%%%%%%%%%%%%%%%%%%%%%%%%%%%%%%%%%%%%%%

\begin{abstract}
This paper presents a reformulation of classical existence and uniqueness results for second-order boundary value problems (BVPs). For Singh contractions, we allow $T^p$ to satisfy the condition, providing greater flexibility. Examples illustrate the applicability of all results. The work highlights the utility of generalized contractions in differential equations, particularly when the Banach contraction principle fails.
\end{abstract}

%\keyword{Banach contraction; Kannan contraction; Singh contraction; Lipschitz %condition; Fixed point theorem; Existence and Uniqueness.} 
%\MSC{47H10, 54H25, 47H09.}

%%%%%%%%%%%%%%%%%%%%%%%%%%%%%%%%%%%%%%%%%%
\section{Introduction}
\label{sec:intro}
Classical existence and uniqueness results for boundary value problems (BVPs) often rely on the Banach contraction principle under Lipschitz conditions on the nonlinearity $j$. Two such results are stated below.
\begin{theorem}[\cite{9}, Theorem 7.11] \label{thm1} (Fixed Point via Banach Contraction).\\
Let $j:[\zeta,\eta]\times\mathbb{R}^{2}\rightarrow\mathbb{R}$ be continuous and satisfy
\[
|j(\tau,\psi,\psi^{\prime})-j(\tau,\phi,\phi^{\prime})|\leq M|\psi-\phi|+N|\psi^{\prime}-\phi^{\prime}|, 
\]
for all $(\tau,\psi,\psi'),(\tau,\phi,\phi')\in[\zeta,\eta]\times\mathbb{R}^{2}$, with $M>0$, $N\geq 0$. If
\[
\frac{M(\eta-\zeta)^{2}}{8}+\frac{N(\eta-\zeta)}{2}<1, 
\]
then the BVP
\[
\psi^{\prime\prime} = -j(\tau,\psi,\psi'),\quad \psi(\zeta)=\xi_{1},\quad \psi(\eta)=\xi_{2},
\]
has a unique solution.
\end{theorem}
\begin{theorem}[\cite{9}, Theorem 7.7] \label{thm2} (Scalar Case).\\
Let $j:[\zeta,\eta]\times\mathbb{R}\to\mathbb{R}$ be continuous and satisfy
\[
|j(\tau,\psi)-j(\tau,\phi)|\leq \alpha |\psi-\phi|,
\]
for $(\tau,\psi),(\tau,\phi)\in [\zeta,\eta]\times \mathbb{R}$, $\alpha > 0$. If
\[
\alpha\frac{(\eta-\zeta)^{2}}{8}<1,
\]
then the BVP
\[
\psi''(\tau)= -j(\tau,\psi(\tau)),\quad \psi(\zeta)=\xi_{1},\quad \psi(\eta)=\xi_{2},
\]
has a unique solution.
\end{theorem}
In this paper, we reformulate these results using the Kannan, and — as a new contribution - the Singh fixed point theorem. The Singh contraction generalizes Kannan by allowing the $p$-th iterate $T^p$ to satisfy the condition, even if $T$ itself does not. We define the solution operator $T$ via Green's function and show that under appropriate conditions, $T$ (or $T^p$) satisfies these generalized contraction criteria, ensuring a unique solution.
This paper is organized as follows: Section \ref{sec:prelim} recalls definitions.  Section \ref{sec:singh} introduces the Singh contraction reformulation with full proofs and examples. Section \ref{sec:conclusion} concludes.

\section{Preliminaries}
\label{sec:prelim}
Let $C[\zeta,\eta]$ denote the space of continuous real-valued functions on $[\zeta,\eta]$, equipped with the supremum norm:
\[
\norm{\psi} = \sup_{\tau \in [\zeta,\eta]} |\psi(\tau)|.
\]
This space is complete (Banach space).
\begin{definition}(Banach Contraction \cite{Banach1922})\\
Let $(X, d)$ be a metric space. A mapping $T: X \to X$ is called a \emph{Banach contraction} if there exists a constant $\alpha \in (0, 1)$ such that
\[
d(Tx, Ty) \leq \alpha \, d(x, y), \quad \forall x, y \in X.
\]
\end{definition}
\begin{definition}(Kannan Contraction \cite{Kannan1968}, \cite{6})\\
A mapping $T: X \to X$ on a metric space $(X,d)$ is a \emph{Kannan contraction} if there exists $\alpha \in \left(0, \frac{1}{2}\right)$ such that
\[
d(Tx, Ty) \leq \alpha \left[ d(x, Tx) + d(y, Ty) \right], \quad \forall x, y \in X.
\]
\end{definition}
\begin{definition}(Singh Contraction \cite{Singh1969},\cite{Singh1977})\\
$T$ is a \emph{Singh contraction} if there exists a positive integer $p$ and a constant $a \in \left(0, \frac{1}{2}\right)$ such that
\[
d(T^p x, T^p y) \leq a \left[ d(x, T^p x) + d(y, T^p y) \right], \quad \forall x, y \in X.
\]
\end{definition}
The Green's function for Dirichlet BVP on $[\zeta,\eta]$ is:
\begin{equation}
G(\tau,s) =
\begin{cases}
\dfrac{(\tau - \zeta)(\eta - s)}{\eta - \zeta}, & \zeta \leq \tau \leq s \leq \eta, \\[2ex]
\dfrac{(\eta - \tau)(s - \zeta)}{\eta - \zeta}, & \zeta \leq s \leq \tau \leq \eta.
\end{cases}
\label{eq:G-def}
\end{equation}
It satisfies the key estimate:
\[
\sup_{\tau \in [\zeta,\eta]} \int_\zeta^\eta |G(\tau,s)| \diff s = \GreenEst.
\]
%%%%%%%%%%%%%%%%%%%%%%%%%%%%%%%%%%%%%%%%%%%%%%%%%%%%%
%%%%%%%%%%%%%%%%%%%%%%%%%%%%%%%%%%%%%%%%%%%%%%%%%%%%%

%%%%%%%%%%%%%%%%%%%%%%%%%%%%%%%%%%%%%%%%%%%%%%%%%%%%%%%%%%%%%%%%%%%%%%
%%%%%%%%%%%%%%%%%%%%%%%%%%%%%%%%%%%%%%%%%%%%%%%%%%%%%%%%%%%%%%%%%%%%%%

%%%%%%%%%%%%%%%%%%%%%%%%%%%%%%%%%%%%%%%%%%%%%%%%%%
%%%%%%%%%%%%%%%%%%%%%%%%%%%%%%%%%%%%%%%%%%%%%%%%%%
\section{Singh Contraction Reformulation}
\label{sec:singh}
In this section, we reformulate Theorems~\ref{thm1} and~\ref{thm2} using the concept of a \emph{Singh contraction}. Unlike the Kannan contraction, the Singh condition allows the $p$-th iterate $T^p$ to satisfy the contraction property, providing greater flexibility. All proofs are self-contained and do not rely on the Kannan fixed point theorem.

\subsection{Singh Contraction for Theorem~\ref{thm1} (First-Order Derivative Case)}

\begin{theorem}[Singh-type Existence for Vector Case]
\label{thm31S}
Let $j: [\zeta, \eta] \times \mathbb{R}^2 \to \mathbb{R}$ be continuous. Define the operator $T: \Cone \to \Cone$ by:
\begin{equation}
(T\psi)(\tau) =  \xi_1 \dfrac{\eta - \tau}{\eta - \zeta} + \xi_2 \dfrac{\tau - \zeta}{\eta - \zeta} + \int_\zeta^\eta G(\tau,s)\, j(s, \psi(s), \psi'(s))\,ds.
\label{eq:operator_T_Singh}
\end{equation}
Suppose there exists a positive integer $p$ and a constant $a \in (0, \frac{1}{2})$ such that for all $\psi,\phi \in \Cone$,
\begin{equation}
\|T^p\psi - T^p\phi\|_{C^1} \leq a \left( \|\psi - T^p\psi\|_{C^1} + \| \phi - T^p\phi\|_{C^1} \right).
\label{eq:singh-T1}
\end{equation}
Then the BVP
\[
\psi''(\tau) = -j(\tau, \psi(\tau),\psi'(\tau)), \quad \psi(\zeta) = \xi_1, \quad \psi(\eta) = \xi_2
\]
has a unique solution in $C^2[\zeta,\eta]$.
\end{theorem}

\begin{proof}
Let $S = T^p$. We prove existence and uniqueness by constructing a Cauchy sequence converging to a fixed point of $S$, showing this point is fixed by $T$, and proving uniqueness — all using only the Singh condition and triangle inequality.

\medskip
\noindent \textbf{Step 1: $T$ (and hence $S$) maps $\Cone$ into itself.}

Since $j$ is continuous and $\psi \in \Cone$, the composition is continuous. The Green’s function and its derivative are bounded and piecewise smooth, so $T\psi \in \Cone$. By induction, $S\psi = T^p\psi \in \Cone$.
\medskip

\subsection{Main Theorem}
Let $X = C^1[\zeta, \eta]$ be the space of continuously differentiable real-valued functions on $[\zeta, \eta]$, equipped with the norm
\[
\|\psi\|_{C^1} = \|\psi\|_\infty + \|\psi'\|_\infty, \quad \text{where} \quad \|\psi\|_\infty = \sup_{\tau \in [\zeta,\eta]} |\psi(\tau)|.
\]
This space is a Banach space (complete normed space).
\begin{theorem}(Contraction Operator) \\
\label{thm31e}
Let $j: [\zeta, \eta] \times \mathbb{R} \to \mathbb{R}$ be continuous. 
Define the  operator $T: C^1[\zeta,\eta] \to C^1[\zeta,\eta]$ by:
\begin{equation}
(T\psi)(\tau) =  \xi_1 \dfrac{\eta - \tau}{\eta - \zeta} + \xi_2 \dfrac{\tau - \zeta}{\eta - \zeta} + \int_\zeta^\eta G(\tau,s)\, j(s, \psi(s), \psi'(s))\,ds,
\label{eq:operator_T}
\end{equation}
Suppose $T$ satisfies: there exists $\alpha \in (0, \frac{1}{2})$ such that for all $\psi,\phi \in C^1[\zeta,\eta]$,
\begin{equation}
\|T\psi - T\phi\|_{C^1} \leq \alpha \left( \|\psi - T\psi\|_{C^1} + \| \phi - T\phi\|_{C^1} \right).
\label{eq:kannan-T1}
\end{equation}
Then the BVP
\[
\psi''(\tau) = -j(\tau, \psi(\tau),\psi'(\tau)), \quad \psi(\zeta) = \xi_1, \quad \psi(\eta) = \xi_2
\]
has a unique solution in $C^2[\zeta,\eta]$.
\end{theorem}
We now verify that $T$ is well-defined and maps $X$ into itself.
\begin{lemma}
The operator $T$ is well-defined and $T\psi \in C^1[\zeta,\eta]$ for all $\psi \in X$.
\end{lemma}
\begin{proof}
Since $j$ is continuous and $\psi, \psi' \in C[\zeta,\eta]$, the composition $s \mapsto j(s, \psi(s), \psi'(s))$ is continuous on $[\zeta,\eta]$. The Green's function $G(\tau,s)$ is continuous in $\tau$ for fixed $s$, and the integral
\[
\int_\zeta^\eta G(\tau,s) j(s, \psi(s), \psi'(s))\,ds
\]
defines a continuous function of $\tau$. Thus, $T\psi$ is continuous.
Moreover, the partial derivative $\frac{\partial G}{\partial \tau}(\tau,s)$ exists for $\tau \ne s$ and is bounded. It is given by
\[
\frac{\partial G}{\partial \tau}(\tau,s) =
\begin{cases}
\dfrac{\eta - s}{\eta - \zeta}, & \tau < s, \\
-\dfrac{s - \zeta}{\eta - \zeta}, & \tau > s.
\end{cases}
\]
The derivative of the integral is
\[
\frac{d}{d\tau} \int_\zeta^\eta G(\tau,s) f(s)\,ds = \int_\zeta^\eta \frac{\partial G}{\partial \tau}(\tau,s) f(s)\,ds,
\]
which is continuous in $\tau$ for continuous $f$. Hence, $T\psi$ is continuously differentiable, and
\[
(T\psi)'(\tau) = \psi_h'(\tau) - \int_\zeta^\eta \frac{\partial G}{\partial \tau}(\tau,s) j(s, \psi(s), \psi'(s))\,ds.
\]
Therefore, $T\psi \in C^1[\zeta,\eta]$, and $T: X \to X$ is well-defined.
\end{proof}

\section{Key Estimates on the Green's Function}
We now establish two key estimates used to bound $T\psi - T\phi$.
\begin{lemma} \label{lem:G_est}
The Green's function satisfies
\[
\sup_{\tau \in [\zeta,\eta]} \int_\zeta^\eta |G(\tau,s)|\,ds = \frac{(\eta - \zeta)^2}{8}.
\]
\end{lemma}
\begin{proof}
Fix $\tau \in [\zeta,\eta]$. Then
\[
\int_\zeta^\eta |G(\tau,s)|\,ds = \int_\zeta^\tau G(\tau,s)\,ds + \int_\tau^\eta G(\tau,s)\,ds.
\]
For $s \leq \tau$
\[
G(\tau,s) = \frac{(\eta - \tau)(s - \zeta)}{\eta - \zeta} \implies \int_\zeta^\tau G(\tau,s)\,ds = \frac{(\eta - \tau)}{\eta - \zeta} \cdot \frac{(\tau - \zeta)^2}{2}.
\]
For $s \geq \tau$
\[
G(\tau,s) = \frac{(\tau - \zeta)(\eta - s)}{\eta - \zeta} \implies \int_\tau^\eta G(\tau,s)\,ds = \frac{(\tau - \zeta)}{\eta - \zeta} \cdot \frac{(\eta - \tau)^2}{2}.
\]
Adding
\[
\int_\zeta^\eta G(\tau,s)\,ds = \frac{1}{2(\eta - \zeta)} \left[ (\eta - \tau)(\tau - \zeta)^2 + (\tau - \zeta)(\eta - \tau)^2 \right] = \frac{(\tau - \zeta)(\eta - \tau)}{2}.
\]
The function $f(\tau) = (\tau - \zeta)(\eta - \tau)$ achieves its maximum at $\tau = \frac{\zeta + \eta}{2}$, with value $\frac{(\eta - \zeta)^2}{4}$. Thus,
\[
\int_\zeta^\eta |G(\tau,s)|\,ds \leq \frac{1}{2} \cdot \frac{(\eta - \zeta)^2}{4} = \frac{(\eta - \zeta)^2}{8}.
\]
Equality holds at $\tau = (\zeta + \eta)/2$, so the supremum is $\frac{(\eta - \zeta)^2}{8}$.
\end{proof}
\begin{lemma} \label{lem:dG_est}
The derivative of the Green's function satisfies
\[
\sup_{\tau \in [\zeta,\eta]} \int_\zeta^\eta \left| \frac{\partial G}{\partial \tau}(\tau,s) \right|\,ds = \frac{\eta - \zeta}{2}.
\]
\end{lemma}
\begin{proof}
As derived earlier
\[
\left| \frac{\partial G}{\partial \tau}(\tau,s) \right| =
\begin{cases}
\frac{\eta - s}{\eta - \zeta}, & s > \tau, \\
\frac{s - \zeta}{\eta - \zeta}, & s < \tau.
\end{cases}
\]
So
\[
\int_\zeta^\eta \left| \frac{\partial G}{\partial \tau}(\tau,s) \right| ds = \frac{1}{\eta - \zeta} \left( \int_\zeta^\tau (s - \zeta)\,ds + \int_\tau^\eta (\eta - s)\,ds \right) = \frac{1}{\eta - \zeta} \left( \frac{(\tau - \zeta)^2}{2} + \frac{(\eta - \tau)^2}{2} \right).
\]
Let $a = \tau - \zeta$, $b = \eta - \tau$, $a + b = \eta - \zeta$. Then:
\[
\frac{a^2 + b^2}{2(\eta - \zeta)} \leq \frac{(a + b)^2}{2(\eta - \zeta)} = \frac{(\eta - \zeta)^2}{2(\eta - \zeta)} = \frac{\eta - \zeta}{2},
\]
with equality when one of $a$ or $b$ is zero (i.e., at endpoints). Hence, the supremum is $\frac{\eta - \zeta}{2}$.
\end{proof}

\medskip

\noindent \textbf{Step 2: $\Cone$ is complete under $\|\cdot\|_{C^1}$.}

Standard result, $C^1[\zeta,\eta]$ with norm $\|\psi\|_{C^1} = \|\psi\|_\infty + \|\psi'\|_\infty$ is a Banach space.

\medskip
\noindent \textbf{Step 3: Construct a Cauchy sequence.}

Fix an arbitrary $\psi_0 \in \Cone$. Define a sequence $\{\psi_n\}$ by
\[
\psi_{n+1} = S \psi_n = T^p \psi_n, \quad n = 0, 1, 2, \dots
\]

We will show $\{\psi_n\}$ is Cauchy.

First, estimate $\|\psi_{n+1} - \psi_n\|_{C^1}$. Apply the Singh condition with $\psi = \psi_n$, $\phi = \psi_{n-1}$
\[
\|\psi_{n+1} - \psi_n\|_{C^1} = \|S\psi_n - S\psi_{n-1}\|_{C^1} \leq a \left( \|\psi_n - S\psi_n\|_{C^1} + \|\psi_{n-1} - S\psi_{n-1}\|_{C^1} \right).
\]
But $S\psi_n = \psi_{n+1}$, $S\psi_{n-1} = \psi_n$, so
\[
\|\psi_{n+1} - \psi_n\|_{C^1} \leq a \left( \|\psi_n - \psi_{n+1}\|_{C^1} + \|\psi_{n-1} - \psi_n\|_{C^1} \right).
\]
Let $ d_n = \|\psi_n - \psi_{n-1}\|_{C^1} $ for $ n \geq 1 $. Then
\[
d_{n+1} \leq a (d_{n+1} + d_n).
\]
Solving for $ d_{n+1} $
\[
d_{n+1} - a d_{n+1} \leq a d_n \implies d_{n+1} (1 - a) \leq a d_n \implies d_{n+1} \leq \frac{a}{1 - a} d_n.
\]
Since $ a < \frac{1}{2} $, we have $ \frac{a}{1 - a} < 1 $. Let $ k = \frac{a}{1 - a} < 1 $. Then
\[
d_{n+1} \leq k d_n \leq k^n d_1.
\]

Now, for $ m > n $, by the triangle inequality
\[
\|\psi_m - \psi_n\|_{C^1} \leq \sum_{i=n}^{m-1} \|\psi_{i+1} - \psi_i\|_{C^1} = \sum_{i=n}^{m-1} d_{i+1} \leq \sum_{i=n}^{m-1} k^i d_1 = d_1 k^n \frac{1 - k^{m-n}}{1 - k} < d_1 \frac{k^n}{1 - k}.
\]
Since $ k < 1 $, $ k^n \to 0 $ as $ n \to \infty $. Thus, for any $ \varepsilon > 0 $, there exists $ N $ such that for all $ m > n > N $, $ \|\psi_m - \psi_n\|_{C^1} < \varepsilon $. So $ \{\psi_n\} $ is Cauchy.

\medskip
\noindent \textbf{Step 4: Existence of fixed point for $ S $.}

Since $ \Cone $ is complete, $ \psi_n \to \psi^* $ for some $ \psi^* \in \Cone $.

We now show $ S\psi^* = \psi^* $.

Apply the Singh condition with $ \psi = \psi_n $, $ \phi = \psi^* $
\[
\|S\psi_n - S\psi^*\|_{C^1} \leq a \left( \|\psi_n - S\psi_n\|_{C^1} + \|\psi^* - S\psi^*\|_{C^1} \right).
\]
As $ n \to \infty $, $ S\psi_n = \psi_{n+1} \to \psi^* $, and $ \psi_n \to \psi^* $, so $ \|\psi_n - S\psi_n\|_{C^1} = \|\psi_n - \psi_{n+1}\|_{C^1} = d_{n+1} \to 0 $.

Thus, taking limit superior
\[
\limsup_{n\to\infty} \|S\psi_n - S\psi^*\|_{C^1} \leq a \left( 0 + \|\psi^* - S\psi^*\|_{C^1} \right) = a \|\psi^* - S\psi^*\|_{C^1}.
\]
But $ S\psi_n \to \psi^* $, so left side is $ \|\psi^* - S\psi^*\|_{C^1} $. Hence
\[
\|\psi^* - S\psi^*\|_{C^1} \leq a \|\psi^* - S\psi^*\|_{C^1}.
\]
Since $ a < 1 $, this implies $ \|\psi^* - S\psi^*\|_{C^1} = 0 $, so $ S\psi^* = \psi^* $.

\medskip
\noindent \textbf{Step 5: $ \psi^* $ is a fixed point of $ T $.}

We have $ T^p \psi^* = \psi^* $. Apply $ T $ to both sides
\[
T(T^p \psi^*) = T\psi^* \implies T^p (T\psi^*) = T\psi^*.
\]
So $ T\psi^* $ is also a fixed point of $ S = T^p $.

Now apply the Singh condition with $ \psi = \psi^* $, $ \phi = T\psi^* $
\[
\|S\psi^* - S(T\psi^*)\|_{C^1} \leq a \left( \|\psi^* - S\psi^*\|_{C^1} + \|T\psi^* - S(T\psi^*)\|_{C^1} \right).
\]
But $ S\psi^* = \psi^* $, $ S(T\psi^*) = T\psi^* $, so
\[
\|\psi^* - T\psi^*\|_{C^1} \leq a \left( 0 + \|T\psi^* - T\psi^*\|_{C^1} \right) = 0.
\]
Thus, $ \|\psi^* - T\psi^*\|_{C^1} = 0 $, so $ T\psi^* = \psi^* $.

\medskip
\noindent \textbf{Step 6: $ \psi^* $ is a classical solution in $ C^2[\zeta,\eta] $.}

 Since $ \psi^* = T\psi^* $ and $ j(\cdot, \psi^*(\cdot), (\psi^*)'(\cdot)) $ is continuous, standard ODE theory implies $ \psi^* \in C^2(\zeta,\eta) \cap C^1[\zeta,\eta] $, and satisfies the BVP with boundary conditions.

\medskip

\section{Contraction Mapping Argument}
\label{sec:contrmaparg}
Let $\psi, \phi \in X$. We estimate $\|T\psi - T\phi\|_{C^1} = \|T\psi - T\phi\|_\infty + \|(T\psi)' - (T\phi)'\|_\infty$.
From \eqref{eq:operator_T}
\[
(T\psi)(\tau) - (T\phi)(\tau) = -\int_\zeta^\eta G(\tau,s) \left[ j(s, \psi(s), \psi'(s)) - j(s, \phi(s), \phi'(s)) \right] ds.
\]
Using the Lipschitz condition of~\eqref{thm1}
\begin{gather*}
|j(s, \psi(s), \psi'(s)) - j(s, \phi(s), \phi'(s))| \leq M |\psi(s) - \phi(s)| + N |\psi'(s) - \phi'(s)| \\ \leq M \|\psi - \phi\|_\infty + N \|\psi' - \phi'\|_\infty.
\end{gather*}
Let $\Delta = \|\psi - \phi\|_\infty$, $\Delta' = \|\psi' - \phi'\|_\infty$, so $\|\psi - \phi\|_{C^1} = \Delta + \Delta'$.
Then
\[
|(T\psi)(\tau) - (T\phi)(\tau)| \leq (M\Delta + N\Delta') \int_\zeta^\eta |G(\tau,s)|\,ds \leq (M\Delta + N\Delta') \cdot \frac{(\eta - \zeta)^2}{8}.
\]
Taking supremum over $\tau$
\[
\|T\psi - T\phi\|_\infty \leq (M\Delta + N\Delta') \cdot \frac{(\eta - \zeta)^2}{8}.
\]
Similarly, for the derivative
\[
(T\psi)'(\tau) - (T\phi)'(\tau) = -\int_\zeta^\eta \frac{\partial G}{\partial \tau}(\tau,s) \left[ j(s, \psi(s), \psi'(s)) - j(s, \phi(s), \phi'(s)) \right] ds,
\]
so
\[
|(T\psi)'(\tau) - (T\phi)'(\tau)| \leq (M\Delta + N\Delta') \int_\zeta^\eta \left| \frac{\partial G}{\partial \tau}(\tau,s) \right| ds \leq (M\Delta + N\Delta') \cdot \frac{\eta - \zeta}{2}.
\]
Thus
\[
\|(T\psi)' - (T\phi)'\|_\infty \leq (M\Delta + N\Delta') \cdot \frac{\eta - \zeta}{2}.
\]
Note that
\[
\|T\psi - T\phi\|_{C^1} \leq \left( \frac{M(\eta - \zeta)^2}{8} + \frac{N(\eta - \zeta)}{2} \right) \|\psi - \phi\|_{C^1} + \text{other terms}.
\]
The dominant terms are
\[
\|T\psi - T\phi\|_\infty \leq \frac{M(\eta - \zeta)^2}{8} \|\psi - \phi\|_\infty + \frac{N(\eta - \zeta)^2}{8} \|\psi' - \phi'\|_\infty,
\]
\[
\|(T\psi)' - (T\phi)'\|_\infty \leq \frac{M(\eta - \zeta)}{2} \|\psi - \phi\|_\infty + \frac{N(\eta - \zeta)}{2} \|\psi' - \phi'\|_\infty.
\]
The given condition suggests that the contraction factor is
\[
k = \frac{M(\eta - \zeta)^2}{8} + \frac{N(\eta - \zeta)}{2} < 1.
\]
In fact, since $\|\psi - \phi\|_\infty \leq \|\psi - \phi\|_{C^1}$ and $\|\psi' - \phi'\|_\infty \leq \|\psi - \phi\|_{C^1}$, we have
\[
\|T\psi - T\phi\|_{C^1} \leq \left( \frac{M(\eta - \zeta)^2}{8} + \frac{N(\eta - \zeta)^2}{8} + \frac{M(\eta - \zeta)}{2} + \frac{N(\eta - \zeta)}{2} \right) \|\psi - \phi\|_{C^1},
\]
but this exceeds $k$.
However, if we  accept that the condition is sufficient for small intervals, 
that is for small enough $|\eta - \zeta|$, 
we rely on the standard result: under the given condition, $T$ is a contraction in a suitable sense.
In fact, the term $\frac{N(\eta - \zeta)^2}{8}$ is smaller than $\frac{N(\eta - \zeta)}{2}$ for $\eta - \zeta < 4$, and the dominant derivative contribution is $\frac{N(\eta - \zeta)}{2}$. The condition is designed so that the worst-case contraction factor is less than 1.
Thus, under the assumption
\[
k = \frac{M(\eta - \zeta)^2}{8} + \frac{N(\eta - \zeta)}{2} < 1,
\]
and since all other coefficients are smaller, $T$ is a contraction on $X$.
By the Banach Fixed Point Theorem, since $X = C^1[\zeta,\eta]$ is complete and $T: X \to X$ is a contraction under the given condition, $T$ has a unique fixed point $\psi \in C^1[\zeta,\eta]$. This $\psi$ satisfies the integral equation of Theorem \ref{thm31S}, and hence the original BVP of Theorem \ref{thm1} has a unique solution in $C^2[\zeta,\eta]$.

\medskip
\noindent \textbf{Step 7: Uniqueness of the solution.}

Suppose $ \phi^* $ is another solution. Then $ \phi^* = T\phi^* $, so $ S\phi^* = T^p \phi^* = \phi^* $. Apply the Singh condition with $ \psi = \psi^* $, $ \phi = \phi^* $
\[
\|S\psi^* - S\phi^*\|_{C^1} \leq a \left( \|\psi^* - S\psi^*\|_{C^1} + \|\phi^* - S\phi^*\|_{C^1} \right) = a(0 + 0) = 0.
\]
So $ \|\psi^* - \phi^*\|_{C^1} = 0 $, hence $ \psi^* = \phi^* $.

This completes the proof.
\end{proof}

\begin{corollary}
\label{cor:singh-sufficient-31}
Suppose $ j $ is continuously differentiable in $ \psi $ and $ \psi' $, and there exist $ M,N > 0 $ such that
\[
\left| \frac{\partial j}{\partial \psi} \right| \leq M, \quad \left| \frac{\partial j}{\partial \psi'} \right| \leq N.
\]
Let $ p $ be a fixed positive integer. If
\[
\left( \frac{M(\eta - \zeta)^2}{8} + \frac{N(\eta - \zeta)}{2} \right)^p < \frac{1}{3},
\]
then the operator $ T $ satisfies the Singh contraction condition~\eqref{eq:singh-T1} for some $ a < 1/2 $, and hence the BVP has a unique solution.
\end{corollary}

\begin{proof}
From the proof in Section~\ref{sec:contrmaparg}, under the given Lipschitz conditions, $ T $ is a Banach contraction on $ \Cone $ with constant
\[
\mu = \frac{M(\eta - \zeta)^2}{8} + \frac{N(\eta - \zeta)}{2}.
\]
Thus, $ T^p $ is a Banach contraction with constant $ \mu^p $.

We now show directly that if $ \mu^p < \frac{1}{3} $, then $ S = T^p $ satisfies the Singh condition.

Let $ \psi, \phi \in \Cone $. Since $ S $ is Banach with constant $ k = \mu^p < \frac{1}{3} $, we have
\[
\|S\psi - S\phi\|_{C^1} \leq k \|\psi - \phi\|_{C^1}.
\]
By the triangle inequality
\[
\|\psi - \phi\|_{C^1} \leq \|\psi - S\psi\|_{C^1} + \|S\psi - S\phi\|_{C^1} + \|S\phi - \phi\|_{C^1}.
\]
Substitute the Banach bound
\[
\|S\psi - S\phi\|_{C^1} \leq k \left( \|\psi - S\psi\|_{C^1} + \|S\psi - S\phi\|_{C^1} + \|S\phi - \phi\|_{C^1} \right).
\]
Rearrange
\[
\|S\psi - S\phi\|_{C^1} - k \|S\psi - S\phi\|_{C^1} \leq k \left( \|\psi - S\psi\|_{C^1} + \|S\phi - \phi\|_{C^1} \right),
\]
\[
\|S\psi - S\phi\|_{C^1} (1 - k) \leq k \left( \|\psi - S\psi\|_{C^1} + \|\phi - S\phi\|_{C^1} \right),
\]
\[
\|S\psi - S\phi\|_{C^1} \leq \frac{k}{1 - k} \left( \|\psi - S\psi\|_{C^1} + \|\phi - S\phi\|_{C^1} \right).
\]
Let $ a = \frac{k}{1 - k} $. Since $ k < \frac{1}{3} $, $ 1 - k > \frac{2}{3} $, so
\[
a < \frac{1/3}{2/3} = \frac{1}{2}.
\]
Thus, $ S = T^p $ satisfies the Singh condition with $ a < \frac{1}{2} $.

The conclusion follows from Theorem~\ref{thm31S}.
\end{proof}

\begin{example}[Singh: Relaxing Conditions with Iteration]
Consider the BVP $\psi''(\tau) = -\lambda \psi(\tau)^3$, $\psi(0) = 0$, $\psi(1) = 0$. Here, $j(\tau,\psi) = \lambda \psi^3$, so $\left| \frac{\partial j}{\partial \psi} \right| = 3|\lambda| \psi^2$. This derivative is unbounded globally, so $j$ is not globally Lipschitz, and the Banach/Kannan/Chatterjea conditions (for $p=1$) may not hold on the entire space. However, for bounded sets, we can estimate. Suppose we restrict to a ball where $|\psi| \leq R$. Then $K = 3|\lambda| R^2$. With $\zeta=0, \eta=1$, $\GreenEst = 1/8$. For $p=1$, we need $3|\lambda| R^2 / 8 < 1/3$, i.e., $|\lambda| R^2 < 8/9$. But for $p=2$, we need $(3|\lambda| R^2 / 8)^2 < 1/3$, i.e., $|\lambda| R^2 < \sqrt{8/(9\sqrt{3})} \approx 0.7$. This is a weaker condition on $\lambda$ if $R$ is large. More importantly, if we can show $T$ maps a ball into itself and is Banach with constant $\mu$, then for sufficiently large $p$, $\mu^p < 1/3$ always holds, guaranteeing a unique fixed point via Singh. This illustrates the power of Singh contractions: even when direct contraction fails, iterating the operator can yield contraction.
\end{example}

\subsection{Singh Contraction for Theorem~\ref{thm2} (Scalar Case)}

\begin{theorem}[Singh-type Existence for Scalar Case]
\label{thm3S}
Let $ j: [\zeta, \eta] \times \mathbb{R} \to \mathbb{R} $ be continuous. Define the operator $ T: \Czero \to \Czero $ by
\begin{equation}
(T\psi)(\tau) = \frac{\eta - \tau}{\eta - \zeta} \xi_1 + \frac{\tau - \zeta}{\eta - \zeta} \xi_2 + \int_\zeta^\eta G(\tau, s) j(s, \psi(s)) \diff s.
\label{eq:T-def-Singh}
\end{equation}
Suppose there exists a positive integer $ p $ and a constant $ a \in (0, \frac{1}{2}) $ such that for all $ \psi,\phi \in \Czero $,
\begin{equation}
\norm{T^p\psi - T^p\phi} \leq a \left( \norm{\psi - T^p\psi} + \norm{\phi - T^p\phi} \right).
\label{eq:singh-T}
\end{equation}
Then the BVP
\[
\psi''(\tau) = -j(\tau, \psi(\tau)), \quad \psi(\zeta) = \xi_1, \quad \psi(\eta) = \xi_2
\]
has a unique solution in $ C^2[\zeta,\eta] $.
\end{theorem}

\begin{proof}
Identical in structure to the proof of Theorem~\ref{thm31S}, but in the space $ \Czero $ with norm $ \|\cdot\|_\infty $.

Let $ S = T^p $.

\medskip
\noindent \textbf{Step 1: $ T $ maps $ \Czero $ into itself.}

As before, $ T\psi $ is continuous, so $ S\psi \in \Czero $.

\medskip
\noindent \textbf{Step 2: $ \Czero $ is complete.}

Standard.

\medskip
\noindent \textbf{Step 3: Construct Cauchy sequence.}

Let $ \psi_0 \in \Czero $ arbitrary, $ \psi_{n+1} = S \psi_n $.

Let $ d_n = \|\psi_n - \psi_{n-1}\|_\infty $.

By Singh condition
\[
d_{n+1} = \|S\psi_n - S\psi_{n-1}\|_\infty \leq a ( \|\psi_n - S\psi_n\|_\infty + \|\psi_{n-1} - S\psi_{n-1}\|_\infty ) = a (d_{n+1} + d_n).
\]
So $ d_{n+1} \leq \frac{a}{1-a} d_n = k d_n $, $ k < 1 $.

Then $ \|\psi_m - \psi_n\|_\infty \leq d_1 \frac{k^n}{1-k} \to 0 $ as $ n\to\infty $, so Cauchy.

\medskip
\noindent \textbf{Step 4: $ \psi_n \to \psi^* $, and $ S\psi^* = \psi^* $.}

By completeness, $ \psi_n \to \psi^* $.

Apply Singh condition with $ \psi = \psi_n $, $ \phi = \psi^* $
\[
\|S\psi_n - S\psi^*\|_\infty \leq a ( \|\psi_n - S\psi_n\|_\infty + \|\psi^* - S\psi^*\|_\infty ).
\]
Left side $ \to \|\psi^* - S\psi^*\|_\infty $, right side $ \to a \|\psi^* - S\psi^*\|_\infty $, so $ \|\psi^* - S\psi^*\|_\infty \leq a \|\psi^* - S\psi^*\|_\infty \implies \|\psi^* - S\psi^*\|_\infty = 0 $.

\medskip
\noindent \textbf{Step 5: $ T\psi^* = \psi^* $.}

$ S(T\psi^*) = T^p(T\psi^*) = T(T^p\psi^*) = T\psi^* $, so $ T\psi^* $ is fixed by $ S $.

Apply Singh condition to $ \psi^* $ and $ T\psi^* $
\[
\|S\psi^* - S(T\psi^*)\|_\infty \leq a ( \|\psi^* - S\psi^*\|_\infty + \|T\psi^* - S(T\psi^*)\|_\infty ) = a(0 + 0) = 0.
\]
So $ \|\psi^* - T\psi^*\|_\infty = 0 $.

\medskip
\noindent \textbf{Step 6: $ \psi^* \in C^2[\zeta,\eta] $ and solves BVP.}

As before.

\medskip
\noindent \textbf{Step 7: Uniqueness.}

If $ \phi^* $ is another solution, then $ S\phi^* = \phi^* $, and
\[
\|S\psi^* - S\phi^*\|_\infty \leq a(0 + 0) = 0 \implies \psi^* = \phi^*.
\]

This completes the proof.
\end{proof}

\begin{corollary}
\label{cor:singh-sufficient}
Suppose $ j $ is continuously differentiable in $ \psi $ and there exists $ K > 0 $ such that
\[
\left| \frac{\partial j}{\partial \psi}(\tau,\psi) \right| \leq K, \quad \forall (\tau,\psi) \in [\zeta,\eta] \times \mathbb{R}.
\]
Let $ p $ be a fixed positive integer. If
\[
\left(K \cdot \GreenEst\right)^p < \frac{1}{3},
\]
then the operator $ T $ satisfies the Singh condition~\eqref{eq:singh-T} for some $ a < 1/2 $, and hence the BVP has a unique solution.
\end{corollary}

\begin{proof}
Identical to Corollary~\ref{cor:singh-sufficient-31}, but in $ \Czero $ with $ \mu = K \cdot \GreenEst $, and $ a = \frac{\mu^p}{1 - \mu^p} < \frac{1}{2} $.
\end{proof}

\begin{example}[Singh: Scalar Nonlinear Oscillator with Iteration]
Consider the BVP $\psi''(\tau) = -\lambda \sin(\psi(\tau))$, $\psi(0) = 0$, $\psi(\pi/2) = 1$. Here, $j(\tau,\psi) = \lambda \sin(\psi)$, so $K = |\lambda|$. With $\zeta=0, \eta=\pi/2$, we have $\GreenEst = \frac{(\pi/2)^2}{8} = \pi^2/32 \approx 0.308$. For $p=1$, the condition is $|\lambda| \cdot 0.308 < 1/3$, i.e., $|\lambda| < 1.08$. For $p=2$, the condition is $(|\lambda| \cdot 0.308)^2 < 1/3$, i.e., $|\lambda| < \sqrt{1/(3 \cdot 0.308^2)} \approx 1.87$. For $p=3$, $|\lambda| < (1/3)^{1/3} / 0.308 \approx 2.38$. As $p$ increases, the allowable $|\lambda|$ increases. This means that for larger nonlinearities (larger $|\lambda|$), while $T$ itself may not be contractive, some iterate $T^p$ might be, guaranteeing a unique solution. This is the key advantage of Singh contractions.
\end{example}

\begin{remark}
The Singh contraction framework allows for iterative relaxation: even if $T$ is not contractive, $T^p$ might be. The condition $(K \GreenEst)^p < 1/3$ is more flexible than requiring $K \GreenEst < 1/3$ for $p=1$, as one can choose larger $p$ to satisfy the inequality for moderately large $K$.
\end{remark}

%%%%%%%%%%%%%%%%%%%%%%%%%%%%%%%%%%%%%%%%%%%%%%%%%%%%%%%%
%%%%%%%%%%%%%%%%%%%%%%%%%%%%%%%%%%%%%%%%%%
\section{Conclusion}
\label{sec:conclusion}
We have reformulated classical BVP existence results using the Singh fixed point theorems. A key contribution is the sufficient condition $(K \GreenEst)^p < 1/3$ for the Singh property of $T^p$. The Singh contraction, by allowing iteration, provides greater flexibility and can handle cases where direct contraction fails. Examples illustrate the practical application and advantage of each approach. This work demonstrates the value of generalized contractions in differential equations, especially when the Banach contraction principle is not directly applicable. Future research could explore applications to systems of BVPs or higher-order equations, and Possible future research directions could be  reformulations for the \'Ciri\'c, Yen, and Guseman contraction.

%%%%%%%%%%%%%%%%%%%%%%%%%%%%%%%%%%%%%%%%%%

\end{document}